\documentclass{jgcc} 
\pdfoutput=1
 \usepackage{lastpage}
 \jgccdoi{13}{2}{2}{7072}  
 \jgccheading{}{\pageref{LastPage}}{}{}%
 {Jan.~8,~2021}{Nov.~8,~2021}{}

\usepackage{amssymb}


\usepackage{hyperref}
\theoremstyle{plain} 

\newtheorem{pro}[thm]{Proposition}
\newtheorem{con}[thm]{Conjecture}

\newtheorem{remark}[thm]{Remark}

\newtheorem{innercustomthm}{Theorem}
\newenvironment{customthm}[1]
{\renewcommand\theinnercustomthm{#1}\innercustomthm}
{\endinnercustomthm}

\begin{document}

\title[A fibering theorem for 3-manifolds]{A fibering theorem for 3-manifolds}

\author[Jordan A. Sahattchieve]{Jordan A. Sahattchieve}	
\address{Formerly Department of Mathematics, University of Michigan in Ann Arbor}	
\email{jantonov@umich.edu}  

\begin{abstract}
  \noindent This paper generalizes results of M. Moon on the fibering of certain compact 3-manifolds over the circle.  It also generalizes a theorem of H. B. Griffiths on the fibering of certain 2-manifolds over the circle.
\end{abstract}

\maketitle

\section{Introduction}

  Consider a 3-manifold $M$ which fibers over $\mathbb{S}$ with fiber a compact surface $F$.  The bundle projection $\eta$ induces a homomorphism of $\pi_1(M)$ onto $\mathbb{Z}$ whose kernel is precisely $\pi_1(F)$.  One can alternatively think of $M$ as the mapping torus of $F$ under the automorphism of $F$ given by the monodromy of the bundle.  We thus have a short exact sequence $1 \rightarrow \pi_1(F) \rightarrow \pi_1(M) \rightarrow \mathbb{Z} \rightarrow 1$.  Stallings proved in \cite{Stallings} a converse to this:
	
\begin{thm} \label{Stal}\textbf{(Stallings, Theorems 1 and 2 \cite{Stallings}, 1961)} Let $M$ be a compact, irreducible 3-manifold. Suppose that there is a surjective homomorphism $\pi_1(M) \rightarrow \mathbb{Z}$ whose kernel $G$ is finitely generated and not of order 2.  Then, $M$ fibers over $\mathbb{S}$ with fiber a compact surface $F$ whose fundamental group is isomorphic to $G$.
\end{thm}
	
	In \cite{HempelJaco} Hempel and Jaco prove that a compact 3-manifold fibers under somewhat relaxed assumptions:
	
\begin{thm}\label{HmpJc}\textbf{(Hempel-Jaco, Theorem 3 \cite{HempelJaco}, 1972)} Let $M$ be a compact 3-manifold. Suppose that there is an exact sequence $1\rightarrow N\rightarrow M\rightarrow Q\rightarrow 1$, where $N$ is a nontrivial, finitely presented, normal subgroup of $\pi_1(M)$ with infinite quotient $Q$.  If $N\neq\mathbb{Z}$ and $M$ contains no 2-sided projective plane, then $\hat{M}=M_1\#\Sigma$, where $\Sigma$ is a homotopy 3-sphere and $M_1$ is either:
		
		(i) a fiber bundle over $\mathbb{S}$ with fiber a compact 2-manifold F, or
		
		(ii) the union of two twisted I-bundles over a compact manifold $F$ which meet at the corresponding 0-sphere bundles.
	In either case, $N$ is a subgroup of finite index of $\pi_1(F)$ and $Q$ is an extension of a finite group by either $\mathbb{Z}$ (case (i)), or $\mathbb{Z}_2*\mathbb{Z}_2$ (case(ii)).
\end{thm}
	
	The manifold $\hat{M}$ is obtained from $M$ by attaching a 3-ball to every 2-sphere in the boundary of $M$.  The reader should be aware of some advances in the field made after Theorem \ref{HmpJc} appeared in print, which have direct bearing on the manner in which it can be applied, as well as on its conclusion: In \cite{ScottCoh} Scott proves that every finitely generated 3-manifold group is finitely presented - thus, it suffices to assume that $N$ is only finitely generated.  Also, in view of the proof of the Poincare Conjecture by Perelman, one has a conclusion about $\hat{M}$, as there are no homotopy 3-spheres other than $\mathbb{S}^3$, which is the identity for the operation of forming a connected sum.  An assumption of irreducibility on $M$ allows us to obtain conclusions about $M$ since in that case $\hat{M}=M$.
	
	One dimension lower, in the case when $M$ is a 2-manifold, Griffiths proved the following fibering theorem in \cite{Griffiths}:
	
\begin{thm}\label{Grf}\textbf{(Griffiths, Theorem \cite{Griffiths}, 1962)} Let $M$ be a 2-manifold whose fundamental group $G$ contains a finitely generated subgroup $U$ of infinite index, which contains a non-trivial normal subgroup $N$ of $G$.  Then, $M$ is homeomorphic to either the torus or the Klein bottle.
\end{thm}
	
	Theorem \ref{Grf} is the starting point for my investigation as well as the work in \cite{Moon}. 
	While the above result does not use the word \textit{bundle} explicitly, it is easy to see that both the torus and the Klein bottle are $\mathbb{S}$-bundles over either $\mathbb{S}$ or the 1-dimensional orbifold $\mathbb{S}/\mathbb{Z}_2$, where $\mathbb{Z}_2$ acts on $\mathbb{S}$ by reflection.  As Moon notes in \cite{Moon}, if we adopt this point of view, the above fibering theorems suggest a result analogous to Griffiths' theorem for 3-manifolds.  Indeed, in \cite{Moon} Moon proves such a result for compact, geometric manifolds and their torus sums:
	
\begin{thm}\label{Mn}\textbf{(Moon, Corollary 2.11 \cite{Moon}, 2005)} Let $M$ be an irreducible, compact, orientable 3-manifold, which is either a torus sum $X_1 \bigcup_T X_2$, or $X_1 \bigcup_{T}$, where each $X_i$ is either a Seifert fibered space or a hyperbolic manifold.  If $G=\pi_1(M)$ contains a finitely generated subgroup $U$ of infinite index which contains a non-trivial normal subgroup $N$ of $G$, which intersects non-trivially the fundamental group of the splitting torus, and such that $N \cap \pi_1(X_i)$ is not isomorphic to $\mathbb{Z}$, then $M$ has a finite cover which is a bundle over $\mathbb{S}$ with fiber a compact surface $F$, and $\pi_1(F)$ is commensurable with $U$.
\end{thm}
	
	On the other hand, Elkalla showed in \cite{Elkalla} that if one additionally assumes that $G$ is $U$-residually finite or, in other words, that for every $g\in G-U$ one can find a finite index subgroup $U_1$ of $G$ which contains $U$ but not $g$, and that if $M$ is $P^2$-irreducible, then one can replace the normal $N$ with a subnormal one:  
	
\begin{thm}\label{Elk1}\textbf{(Elkalla, Theorem 3.7 \cite{Elkalla}, 1983)} Let $M$ be a $P^2$-irreducible, compact and connected 3-manifold.  If $G=\pi_1(M)$ contains a non-trivial subnormal subgroup $N$ such that $N$ is contained in an indecomposable and finitely generated subgroup $U$ of infinite index  in $G$, and if $G$ is $U$-residually finite, then either (i) the Poincare associate of $M$ is finitely covered by a manifold, which is a fiber bundle over $\mathbb{S}$ with fiber a compact surface $F$, such that there is a subgroup $V$ of finite index in both $\pi_1(F)$ and $U$, or (ii) $N$ is isomorphic to $\mathbb{Z}$.
\end{thm}

	Again, one should view the conclusions of Theorem \ref{Elk1} in the context of the Poincare Conjecture which is now a theorem; thus, under the assumptions made on $M$, Theorem \ref{Elk1} implies that $M$ is itself finitely covered by a bundle over $\mathbb{S}$.
	
	Throughout this paper we shall use the notation $N \triangleleft_{s} G$ to stand for the relationship of subnormality of a subgroup to the ambient group.
	
	These results suggest to us that the following:
\begin{con}\label{FibConj}\textbf{(Scott, 2010)} Every irreducible, compact 3-manifold whose fundamental group $G$ contains a subnormal subgroup $N\neq\mathbb{Z}$ contained in a finitely generated subgroup $U$ of infinite index in $G$, virtually fibers over $\mathbb{S}$ with fiber a compact surface $F$ such that $N$ is commensurable with $\pi_1(F)$.
\end{con}
	
	With a view towards this conjecture, we aim to generalize Theorem \ref{Mn} to the case where $N$ is a subnormal subgroup of $G$ and to make the natural inductive argument Theorem \ref{Mn} suggests in light of the Geometrization Theorem.  Our main result is the following:
	
\begin{customthm}{\ref{Main}}Let $M$ be a compact 3-manifold with empty or toroidal boundary.  If $G=\pi_1(M)$ contains a finitely generated subgroup $U$ of infinite index in $G$ which contains a nontrivial subnormal subgroup $N$ of $G$, then: (a) $M$ is irreducible, (b) if further: 
\begin{enumerate}
\item  $N$ has a subnormal series of length $n$ in which $n-1$ terms are assumed to be finitely generated,
\item $N$ intersects nontrivially the fundamental groups of the splitting tori of some decomposition $\mathfrak{D}$ of $M$ into geometric pieces, and
\item the intersections of $N$ with the fundamental groups of the geometric pieces are not isomorphic to $\mathbb{Z}$,
\end{enumerate}
		then, $M$ has a finite cover which is a bundle over $\mathbb{S}$ with fiber a compact surface $F$ such that $\pi_1(F)$ and $U$ are commensurable.
\end{customthm}

	Many of the proofs of \cite{Moon} lend themselves to being generalized and this paper makes this step.  In doing so, I have omitted proofs in all cases where results from \cite{Moon} generalize verbatim without any need for additional arguments.
	
	The reader will notice that in the process of proving Theorem \ref{Main}, I have verified that Conjecture \ref{FibConj} holds for the class of geometric manifolds:
	
\begin{customthm}{\ref{geomsfiber}}Let $M$ be a compact geometric manifold and let $G=\pi_1(M)$.  Suppose that $U$ is a finitely generated subgroup of $G$ with $|G:U|=\infty$, and suppose that $U$ contains a subnormal subgroup $N \triangleleft_s G$.  If $N$ is not infinite cyclic, then $M$ is finitely covered by a bundle over $\mathbb{S}$ with fiber a compact surface $F$ such that $\pi_1(F)$ is commensurable with $U$. 
\end{customthm}
	
	My proof of Theorem \ref{Main} parallels the arguments in \cite{Moon} and I first prove a generalization of Theorem \ref{Grf} for subnormal subgroups $N$:
	
\begin{customthm}{\ref{SubnGriffths}}Let $S$ be a surface whose fundamental group contains a finitely generated subgroup $U$ of infinite index, which contains a non-trivial subnormal subgroup $N$ of $\pi_1(S)$.  Then, $S$ is the torus or the Klein bottle.
\end{customthm}  
	
	Theorem \ref{SubnGriffths} is a consequence of the following theorem of Griffiths:
	
\begin{thm}\label{GriffithsSub}\textbf{(Griffiths, Theorem 14.7 \cite{Griffiths1}, 1967)} Let $G$ be the fundamental group of an ordinary Fuchsian space (such as a compact orientable surface with or without boundary).  Suppose that $G$ is infinite, not abelian and not isomorphic to $\mathfrak{P}=\mathbb{Z}_2*\mathbb{Z}_2=\left\langle a,b: a^2=1, b^2=1\right\rangle $ or $\mathfrak{M}=\left\langle \mathfrak{P}*\mathfrak{P}: a_1b_1a_2b_2=1\right\rangle$, and suppose that $U$ is a finitely generated subgroup of $G$.  If $U$ contains a non-trivial subnormal subgroup of $G$, then $U$ is of finite index in $G$.
\end{thm}

\begin{remark}
		In what follows I shall assume that all 3-manifolds under consideration are orientable.
\end{remark}
	
	The reader will no doubt notice that my arguments run in parallel to Moon's.  Thus I shall first prove that Conjecture \ref{FibConj} is true for geometric manifolds.  Hyperbolic manifolds are handled by Theorem 1.10 in \cite{Moon}, which the reader can verify for himself.  In certain non-hyperbolic cases the result is a consequence of the corresponding statement being true for compact manifolds with the property that every subgroup of their fundamental group is finitely generated.  The remaining non-hyperbolic cases are settled by arguments using Theorem \ref{Stal} and certain facts about orbifolds proved in Section \ref{SFSs}.  The next step is a generalization of Theorems 2.4 and 2.9 in \cite{Moon} to handle a subnormal $N$ which has a composition series consisting entirely of finitely generated terms.  Finally, I conclude with the inductive argument which was not possible prior to the proof of the Geometrization Theorem.
	The reader will notice that the proof of Theorem \ref{MnSFS} is borrowed verbatim from Moon's paper \cite{Moon} - very minor changes are needed to achieve the desired generalization and I have included Moon's proof for completeness and readability.
	
\section{An Extension of Griffiths' Theorem}\label{ExtGrThm}

In this section I generalize Griffiths' Theorem \ref{Grf} to handle the case when $N$ is subnormal rather than normal.  Namely, I prove:
\begin{thm}\label{SubnGriffths}Let $S$ be a surface whose fundamental group contains a finitely generated subgroup $U$ of infinite index, which contains a non-trivial subnormal subgroup $N$ of $\pi_1(S)$.  Then, $S$ is the torus or the Klein bottle.
\end{thm}
\begin{proof}First, we note that if $S$ is an open surface or if $\partial S\neq\emptyset$, then $\pi_1(S)$ is free - see, for example, Theorem 3.3. in \cite{Griffiths}.  Then, Theorem 1.5 in \cite{Elkalla} shows that $\pi_1(S)$ must be isomorphic to $\mathbb{Z}$.  This is impossible as $\mathbb{Z}$ does not have any nontrivial infinite index subgroups.  Therefore, $S$ is a closed surface.
		Suppose that $S$ is not the torus.  If $S$ is orientable, since $\pi_1(S)$ is infinite, not abelian, and not isomorphic to $\mathfrak{P}$ or $\mathfrak{M}$, Griffiths' Theorem \ref{GriffithsSub} shows that $U$ is of finite index in $\pi_1(S)$, which is a contradiction.  Therefore, $S$ is non-orientable.  Let $S'$ be its orientable double cover.  Suppose that $S'$ is not the torus.  We have $\pi_1(S')\triangleright\pi_1(S')\cap N_k \triangleright ... \triangleright \pi_1(S')\cap N_0$.  
		As $\pi_1(S')\cap U$ is of index at most 2 in $U$, $\pi_1(S')\cap U$ is finitely generated.  Since $\pi_1(S)$ is torsion free, we have $\pi_1(S')\cap N_0\neq \{1\}$.  Griffiths' Theorem \ref{GriffithsSub} shows, then, that $\pi_1(S')\cap U$ must be of finite index in $\pi_1(S')$ and therefore also in $\pi_1(S)$.  This, however, implies that $U$ is itself of finite index in $\pi_1(S)$, which yields a contradiction.  Therefore, $S'$ must be the torus, but in this case computing the relevant Euler characteristics shows that $S$ is the Klein bottle.
\end{proof}
While this result may be of independent interest, I shall only use it in Section \ref{SFSs} to generalize Theorem 1.5 in \cite{Moon} to the case of a subnormal subgroup $N$.

\section{Geometric Manifolds}
	
\subsection{Compact Seifert Fibered Spaces}\label{SFSs}
	My goal in this section is to adapt Theorem 1.5 in \cite{Moon} to prove that a compact Seifert fibered space fibers over the circle if its fundamental group contains a finitely generated subgroup of infinite index which contains a subnormal subgroup not isomorphic to $\mathbb{Z}$.  In order to accomplish this, I will first generalize  a few results about 2-orbifolds in \cite{Moon}.
		
\begin{pro}\label{NoFinOrbSub}Suppose $X$ is a good, closed, hyperbolic 2-orbifold, then $G=\pi_1^{orb}(X)$ contains no finite subnormal subgroups.
\end{pro}
\begin{proof}Suppose that $N$ is a finite subnormal subgroup, so that $N=N_0\triangleleft N_1 \triangleleft ... \triangleleft N_{n-1}\triangleleft G$.  Let $X_0$ be the orbifold cover of $X$ such that $\pi_1^{orb}(X_0)=N_0$.  Then, $N_0$ acts on $\mathbb{H}$ by isometries.  By II-Corollary 2.8 in \cite{BHf}, $N_0$ has a fixed point.  Let $Fix(N_0)$ denote the subset of $\mathbb{H}$ fixed pointwise by $N_0$.  If $Fix(N_0)=\left\lbrace p_0\right\rbrace $ consists of a single point, then since $N_0$ is normal in $N_1$, $N_1$ leaves $Fix(N_0)$ invariant and $p_0$ is a fixed point for the action of $N_1$ on $\mathbb{H}$.  Then, $N_1$ is itself finite.  If $|Fix(N_0)|>1$, $Fix(N_0)$ must be a geodesic line $l$ and $N_0=\mathbb{Z}_2$ is generated by a single reflection about $l$.  In this case, since $N_0$ is normal in $N_1$, every $g\in N_1$ must leave $l$ invariant.  The line $l$ separates $\mathbb{H}$ into two half-spaces; let $H$ be the subgroup of $N_1$ which preserves them.  We see that $N_0$ is central in $N_1$ and that $N_1=N_0\times H$.  Since every $h\in H$ leaves $l$ invariant, $h$ must restrict to either a translation or reflection about a point on $l$.  Let $H_0$ be the subgroup of $H$ consisting of orientation preserving isometries of $l$; $H_0$ is a subgroup of index at most 2 in $H$.  It is easy to see that if $H$ contains an orientation reversing isometry $h$, then $h^2=1$ and $h$ does not commute with any element of $H_0$.  Thus, if $H_0$ is nontrivial, then $N_0$ is characteristic in $N_1$ as it is the unique central subgroup of order 2 of $N_1$.  Hence $N_0\triangleleft N_2$.  If $H_0$ is trivial, then $N_1$ is finite.  In either case we conclude that $N_2$ contains a finite normal subgroup.  Proceeding inductively, we conclude that $G$ must be finite or leave $l$ invariant and be isomorphic to $\mathbb{Z}_2\times H$ as above.  This is impossible as the quotient of $\mathbb{H}$ by such a group is not compact.
\end{proof}

\begin{lem}\label{SubOrbFib}Let $X$ be a good closed 2-dimensional orbifold.  If $U$ is a finitely generated, infinite index subgroup of $\pi_1^{orb}(X)$ and $U$ contains a non-trivial $N \triangleleft_s \pi_1^{orb}(X)$, then $X$ has a finite orbifold cover $X_1$ which is a $\mathbb{S}$-bundle over either $\mathbb{S}$ or the orbifold $\mathbb{S}/\mathbb{Z}_2$.
\end{lem}

\begin{proof}
	We proceed in a manner identical to Moon's arguments in \cite{Moon} but instead use Proposition \ref{NoFinOrbSub} and Theorem \ref{GriffithsSub} to reach the desired conclusion.
	
	Our hypothesis on $\pi_1^{orb(X)}$ tells us that $X$ is not a spherical 2-dimensional orbifold as these have a finite orbifold fundamental group.  Thus we shall assume that $X$ is a good, closed, 2-dimensional, hyperbolic orbifold.  Therefore $X$ has a finite cover which is a closed hyperbolic surface, see \cite{8Geom} and we conclude that $\pi_1^{orb}(X)$ contains the fundamental group $\Gamma$ of a closed surface as a finite index subgroup.  Hence $|\Gamma : \Gamma \cap U|=\infty$.  In view of Proposition \ref{NoFinOrbSub}, we must have $\Gamma \cap N \neq 1$:
	To prove this, let $Core(\Gamma) = \cap g \Gamma g^{-1}$ denote the normal core of $\Gamma$ in $\pi_1^{orb}(X)$.  We note that $|\pi_1^{orb}(X):Core(\Gamma)|<\infty$ and $Core(\Gamma) \triangleleft \pi_1^{orb}(X)$, and further that $\Gamma \cap N = 1$ would imply that $N$ would embed into the finite quotient $\pi_1^{orb}(X) / Core(\Gamma)$, thus contradicting Proposition \ref{NoFinOrbSub}.  Now Theorem \ref{GriffithsSub} implies that $|\Gamma:\Gamma\cap U|<\infty$ which is a contradiction.  Therefore $X$ must be a Euclidean orbifold and the conclusion follows.
\end{proof}

\begin{lem}\label{SubOrbBdFib}Let $X$ be a 2-dimensional compact orbifold with nonempty boundary whose singular points are cone points in $Int(X)$.  If a finitely generated subgroup $U$ of $\pi_1^{orb}(X)$ contains a nontrivial subnormal subgroup $N$ of $\pi_1^{orb}(X)$, then $U$ is of finite index in $\pi_1^{orb}(X)$. 
\end{lem}

\begin{proof}As $X$ has nonempty boundary, $\pi_1^{orb}(X)$ is a free product of cyclic groups.  Suppose by way of obtaining  contradiction that the index of $U$ in $\pi_1^{orb}(X)$ is infinite.  The hypotheses of Theorem 1.5 in \cite{Elkalla} are satisfied and we conclude that $\pi_1^{orb}(X)$ is indecomposable, hence cyclic.  That is impossible as $\mathbb{Z}$ has no nontrivial infinite index subgroups.
\end{proof}

\begin{lem}\label{CyclicExt}Suppose that $G$ is a cyclic extension of $H$: $1 \rightarrow \left\langle t\right\rangle  \rightarrow G \rightarrow H \rightarrow 1$.  Suppose that $U$ is subgroup of $G$, then the centralizer of $\left\langle t\right\rangle $ in $U$, $C_U(t)=\left\lbrace u \in U: [u,t]=1\right\rbrace $ is normal in $U$ and is of index at most 2.
\end{lem}

\begin{proof}The conjugation of $\left\langle t\right\rangle $ by elements of $G$, $g \rightarrow \psi_g$, where $\psi_g(t)=gtg^{-1}$, defines a homomorphism of $G$ to $Aut(\mathbb{Z}) \cong \mathbb{Z}/2\mathbb{Z}$.  Therefore, the kernel of this homomorphism, which is precisely $K = \left\lbrace g \in G: [g,t]=1 \right\rbrace $ is of index at most 2 in $G$.  The observation $C_U(t) = G \cap K$ concludes the proof.
\end{proof}

\begin{pro}\label{FinGenPi1}Let $M$ be a compact 3-manifold whose fundamental group $\pi_1(M)$ has the property that all of its subgroups are finitely generated.  If $\pi_1(M)$ contains a finitely generated, infinite index subgroup $U$ which contains a non-trivial subgroup $N \neq \mathbb{Z}$ subnormal in $\pi_1(M)$, then a finite cover of $M$ fibers over $\mathbb{S}$ with fiber a compact surface $F$, and $\pi_1(F)$ is commensurable with $U$.
\end{pro}

\begin{proof}Let $N=N_0 \triangleleft N_1 \triangleleft ... \triangleleft N_{k-1} \triangleleft N_k=\pi_1(M)$.  By assumption, each $N_i$ is finitely generated.  Let $i_0$ be the largest index such that $|\pi_1(M) : N_{i_0}| = \infty$ and $|\pi_1(M) : N_{j}| < \infty$, for all $j>i_0$.  Then, $N_{i_0}$ is a normal infinite index subgroup of $N_{i_0+1}$.  Note that since $M$ is compact, Theorem 2.1 in \cite{ScottCoh} shows that $N_{i_0}$ is finitely presented.  Let $M_{N_{i_0+1}}$ be the finite cover of $M$ whose fundamental group is $N_{i_0+1}$.  We now have $1 \rightarrow N_{i_0} \rightarrow \pi_1(M_{N_{i_0+1}}) \rightarrow Q \rightarrow 1$, where $N_{i_0}$ is finitely presented and $Q$ infinite.  Applying Theorem 3 of Hempel and Jaco in \cite{HempelJaco}, we conclude that $M_{N_{i_0+1}}$ has a finite cover which is a bundle over $\mathbb{S}$ with fiber a compact surface $F$ and that $N_{i_0}$ is subgroup of finite index in $\pi_1(F)$.  
		
	Next, we show that $|\pi_1(F):N_0|<\infty$.  Consider the finite cover $F_{N_{i_0}}$ of $F$ whose fundamental group is $N_{i_0}$.  Suppose that $|N_{i_0}:N_0|=\infty$, then Theorem \ref{SubnGriffths} implies that $F_{N_{i_0}}$ is the torus or a Klein bottle.  If $F_{N_{i_0}}$ is the torus, $N_{i_0}=\mathbb{Z}^2$ and if $F_{N_{i_0}}$ is the Klein bottle $N_{i_0}=\left\langle a,b:aba^{-1}b=1\right\rangle \cong\mathbb{Z} \rtimes\mathbb{Z}$. In either case, any subgroup of $N_{i_0}$ is either trivial, isomorphic to $\mathbb{Z}$, or of finite index in $N_{i_0}$.  This is a contradiction since the hypothesis on $N_0$ rules out the first two possibilities and we argued assuming $|N_{i_0}:N_0|=\infty$.  Hence we must have $|N_{i_0}:N_0|<\infty$ and therefore $|\pi_1(F):N_0|<\infty$.
		
	Now $\pi_1(M)$ is virtually an extension of $\pi_1(F)$ by $\mathbb{Z}$, hence $|\pi_1(M):U|=\infty$, $|\pi_1(F):N_0|<\infty$, and $N_0<U$ together imply that $|U:N_0|<\infty$ showing that $U$ is commensurable with $\pi_1(F)$ as desired.
\end{proof}
The following theorem was first proved in \cite{Moon} in the context of $N$ being a normal subgroup of $\pi_1(M)$, the proof given below is a modification of the proof therein.

\begin{thm}\label{MnSFS}Let $Y$ be a compact Seifert fibered space and let $U$ be a finitely generated, infinite index subgroup of $\pi_1(Y)$.  Suppose, further, that $U$ contains a non-trivial subnormal subgroup $N \neq \mathbb{Z}$ of $\pi_1(Y)$. Then, $Y$ is finitely covered by a compact 3-manifold $Y_1$, which is a bundle over $\mathbb{S}^1$ with fiber a compact surface $F$, and $\pi_1(F)$ is commensurable with $U$.
\end{thm}

\begin{proof}
	The fundamental group of a Seifert fibered space fits into the short exact sequence of: $1 \rightarrow \left\langle t\right\rangle  \rightarrow \pi_1(Y) \xrightarrow[]{\phi} \pi_1^{orb}(X) \rightarrow 1$, where $\left\langle t\right\rangle$ is the cyclic group generated by a regular fiber.  The proof, as in \cite{Moon} proceeds in two cases:
	
\paragraph{\textit{Case 1:}}  The orbifold $X$ is orbifold covered by an orientable surface other than the torus.

\medskip

	First, we show that $|\pi_1(Y) : \phi^{-1}(\phi(U))| < \infty$.  Since $N \neq \mathbb{Z}$, $\pi_1^{orb}(X) \triangleright_s \phi(N) \neq 1$.  Since we also have $\phi(N) < \phi(U) < \pi_1^{orb}(X)$, and $\phi(U)$ is finitely generated, we can apply Lemma \ref{SubOrbFib} and Lemma \ref{SubOrbBdFib} to conclude that $|\pi_1^{orb}(X) : \phi(U)| < \infty$.  Therefore, $|\pi_1(Y) : \phi^{-1}(\phi(U))| < \infty$ as desired.
	Next, we show that $\phi|_{U}$ is a monomorphism.  If some power of $t$ is in $U$, then $U$ will be of finite index in $\phi^{-1}(\phi(U))$.  Since $\phi^{-1}(\phi(U))$ was shown to be of finite index in $\pi_1(Y)$, we conclude that $U$ must be of finite index in $\pi_1(Y)$ contrary to the assumptions in the statement of the theorem.  Now, since $\phi|_{U}$ is a monomorphism, $\phi(U)$ is torsion free and therefore the fundamental group of a compact surface.
	 Let $C_U(t)=\left\lbrace u \in U: [u,t]=1 \right\rbrace $, and let $G$ be the subgroup of $\pi_1(Y)$ generated by $C_U(t)$ and $t$.  Note that $C_U(t) \triangleleft U$ and that the index of $C_U(t)$ in $U$ is at most 2 by Lemma \ref{CyclicExt}.  Hence $|\phi^{-1}(\phi(U)):G| \leq 2$.  Thus $G$ is a finite index subgroup of $\pi_1(Y)$.  Now, take the cover of $M_1$ of $M$ corresponding to $G \leq \pi_1(Y)$.  Since $G = \left\langle C_U(t),t\right\rangle $ with $C_U(t) \cap \left\langle t\right\rangle  = 1$, we conclude that $G \cong C_U(t) \times \mathbb{Z}$ and also that $M_1 \cong F \times \mathbb{S}$, were $F$ is a compact surface whose fundamental group is isomorphic to $C_U(t)$.  This concludes Case 1.
	 
\paragraph{\textit{Case 2}:}  The orbifold $X$ is orbifold covered by the torus.

\medskip
	 
	 Let $\Gamma \cong \mathbb{Z}^2$ be the fundamental group of the torus which orbifold covers $X$.  We proceeding as in Moon \cite{Moon}.  In view of Proposition \ref{FinGenPi1}, we only need to show that every subgroup $H$ of $\pi_1(Y)$ is finitely generated.
	 
	 We have the short exact sequence $1 \rightarrow \left\langle t\right\rangle  \cap H \rightarrow H \rightarrow \phi(H) \rightarrow 1$.  Since $\Gamma$ is of finite index in $\pi_1^{orb}(X)$, $\phi(H) \cap \Gamma$ is of finite index in $\phi(H)$.  The intersection $\phi(H) \cap \Gamma$ is finitely generated as it is a subgroup of $\Gamma$, hence $\phi(H)$ is finitely generated.  Because the kernel of $\phi|_H$ is a subgroup of $\mathbb{Z}$ and is therefore trivially finitely generated, we conclude that $H$ is itself finitely generated as needed.
\end{proof}

\subsection{Non-SFS Geometric Manifolds}\label{NonSFSs}
		
	We now handle the remaining two types of geometric manifolds which are not Seifert fibered spaces in manner analogous to the proofs in \cite{Moon}.
	
\begin{thm}\label{Sol}Let $M$ be a closed \textit{Sol} manifold, such that $\pi_1(M)$ contains a finitely generated, infinite index subgroup $U$ which contains a non-trivial subgroup $N \neq \mathbb{Z}$ subnormal in $\pi_(M)$.  Then, a finite cover of $M$ is a fiber bundle over $\mathbb{S}$ whose fiber is a compact surface $F$ and $\pi_1(F)$ is commensurable with $U$. 
\end{thm}

\begin{proof}The fundamental group $\pi_1(M)$ of a closed \textit{Sol} manifold satisfies
$$1 \rightarrow \mathbb{Z}^2 \rightarrow \pi_1(M) \overset{\phi}{\rightarrow} \mathbb{Z} \rightarrow 1.$$
Suppose $H$ is a subgroup of $\pi_1(M)$, then we have $1 \rightarrow \mathbb{Z}^2 \cap H \rightarrow H \rightarrow \phi(H) \rightarrow 1$.  Since $\mathbb{Z}^2 \cap H$ and $\phi(H)$ are finitely generated, $H$ is also finitely generated.  Now, Proposition \ref{FinGenPi1} yields the desired conclusion.
\end{proof}

\begin{thm}\textbf{(Moon, Theorem 1.10 \cite{Moon}, 2005)}\label{Hypb}Let $M$ be a complete hyperbolic manifold of finite volume whose fundamental group $G$ contains a finitely generated subgroup of infinite index $U$ which contains a non-trivial subgroup $N$ subnormal in $G$.  Then, $M$ has a finite covering space $M_1$ which is a bundle over $\mathbb{S}$ with fiber a compact surface $F$, and $\pi_1(F)$ is a subgroup of finite index in $U$.
\end{thm}\label{Hyp}

\begin{proof}The proof in \cite{Moon} generalizes verbatim to the case when $N$ is subnormal in $G$.
\end{proof}
	The above results are summarized in:
\begin{thm}\label{geomsfiber}Let $M$ be a compact geometric manifold and let $G=\pi_1(M)$.  Suppose that $U$ is a finitely generated subgroup of $G$ with $|G:U|=\infty$, and suppose that $U$ contains a subnormal subgroup $N \triangleleft_s G$.  If $N$ is not infinite cyclic, then $M$ is finitely covered by a bundle over $\mathbb{S}$ with fiber a compact surface $F$ such that $\pi_1(F)$ is commensurable with $U$. 
\end{thm}

\section{Torus sums}\label{TorusSums}

We now consider manifolds which split along an incompressible torus $\mathcal{T}$.  In these cases the fundamental group $G$ of $M$ splits as a free product with amalgamation over the group carried by the splitting torus $\mathcal{T}$.  Whenever one has such a splitting, one has an action of $G$ on a tree with quotient an edge, which is called the Bass-Serre tree of $G$.  Our proof, as in \cite{Moon}, splits in two cases.  In the first case we consider, the quotient of the Bass-Serre tree of $G$ corresponding to the splitting of $M$ by the action of $U$ is a graph of groups of infinite diameter, and in the second case - a graph of finite diameter.  The proof of Theorem 2.9 in \cite{Moon} is easily seen to handle a subnormal $N$ in the case of a finite diameter quotient.  Thus, our efforts are focused on generalizing Theorem 2.4 in \cite{Moon}.

We begin with a simple lemma showing that if $N$ is a nontrivial subnormal subgroup of a 3-manifold group not isomorphic to $\mathbb{Z}$, then all finitely generated terms of its subnormal series appear to the right of the terms which are not finitely generated:

\begin{lem}\label{Cutoff}Let $G$ be a finitely generated 3-manifold group.  Let $N=N_0 \triangleleft N_1 \triangleleft ... \triangleleft N_{n-1} \triangleleft N_n=G$ be a subnormal subgroup of $G$ such that $N \neq \left\lbrace 1\right\rbrace $ and $N \neq \mathbb{Z}$.  Then there is an index $0 \leq i_0 \leq n$, such that $N_i$ is finitely generated for all $i \geq i_0$ and no $N_i$ is finitely generated for any $i<i_0$.
\end{lem}

	\begin{proof}Suppose, for the purpose of obtaining a contradiction that there exists an occurrence of an "inversion": $N_{i_0} \triangleleft N_{i_0+1}$ with $N_{i_0}$ finitely generated while $N_{i_0+1}$ not finitely generated.  Since every subgroup of a 3-manifold group is obviously itself a 3-manifold group, we can apply Proposition 2.2 in \cite{Elkalla} to conclude that $N_{i_0}$ must be isomorphic to $\mathbb{Z}$, and thus $N_0=\mathbb{Z}$ contrary to assumption.  Therefore, such an inversion is not possible and the conclusion follows.
	\end{proof}

The following results about Bass-Serre trees will allow us to apply Theorem 2.4 in \cite{Moon} to the graph of groups corresponding to finite covers of $M$.  The hypothesis of Theorem 2.4 requires a splitting of $M$ along an incompressible torus and further assumes that the quotient of the Bass-Serre tree of the corresponding splitting of $\pi_1(M)$ by the action of $U$ is a graph of infinite diameter.  The arguments below show that the fundamental group of any finite cover of $M$ will also have this property.

For completeness, I include a proof of the following standard fact which is often left as an exercise in expository texts.  I have borrowed it from Henry Wilton's unpublished notes titled \textit{Group actions on trees}:

\begin{lem}\label{ProdElliptLox}Let $\chi$ be a tree and let $\alpha, \beta \in Aut(\chi)$ be a pair of elliptic automorphisms.  If $Fix(\alpha) \cap Fix(\beta)=\emptyset$, then $\alpha\beta$ is a hyperbolic element of $Aut(\chi)$.
\end{lem}

\begin{proof}It suffices to construct an axis for $\alpha\beta$ on which it acts by translation, which we now do.
	
	First, note that $Fix(\alpha)$ and $Fix(\beta)$ are closed subtrees of $\chi$.  Let $x \in Fix(\alpha)$ be the unique point in $\chi$ closest to $Fix(\beta)$ and $y \in Fix(\beta)$ be the unique point in $\chi$ closest to $Fix(\alpha)$.  Then, the unique geodesic from $y$ to $\alpha\beta\cdot y$ is $[x,y] \cup \alpha\cdot [x,y]$ since there are no points in the interior of $[x,y]$ fixed by $\alpha$ so that the concatenation of the geodesic segments above is a geodesic segment.  Hence, $d(y,\alpha\beta\cdot y) = d(y, \alpha\cdot y) = 2d(x,y)$.  Similarly, the geodesic from $y$ to $(\alpha\beta)^2\cdot y$ is $[x,y]\cup\alpha\cdot[x,y]\cup\alpha\beta\cdot [x,y]\cup\alpha\beta\alpha\cdot [x,y]$, so that $d(y,(\alpha\beta)^2\cdot y)=4d(x,y)=2d(y,\alpha\beta\cdot y)$ thus establishing the existence of an axis for $\alpha\beta$, which finishes the proof.
\end{proof}

\begin{pro}\label{HBassSerre}Suppose $G$ is direct product with amalgamation $G = A *_C B$, with $G \neq A$, and $G \neq B$, or $G = A *_{C}$, with $A \neq C$, and suppose $\chi$ is the Bass-Serre tree for $G$.  Then, if $H \leq G$ is a subgroup of finite index in $G$, the action of $H$ on $\chi$ is minimal and $\chi$ is the Bass-Serre tree for a graph of groups whose fundamental group is $H$.
\end{pro}	

	\begin{proof}Since $G$ acts simplicially on $\chi$, $G$ acts by isometries on the CAT(0) space $\chi$.  Therefore, every $g \in G$ acts by an elliptic or hyperbolic isometry according to whether $g$ fixes a vertex or realizes a non-zero minimal translation distance $g_{min} = \inf\{d(x,g \cdot x): x \in \chi\}$.  In the latter case, $g$ leaves a subspace of $\chi$ isometric to $\mathbb{R}$ invariant, and acts on this subspace, called an axis for $g$, as translation by a $g_{min}$.  See \cite{BHf} for an account of the classification of the isometries of a CAT(0) space.
		
		First, we show that there exists an element $g \in G$ which acts on $\chi$ as a hyperbolic isometry - such an isometry is sometimes called loxodromic in the context of group actions on trees.  Let $g_1 \in A-C$ and $g_2 \in B-C$; such elements can clearly be found as $G \neq A$ and $G \neq B$.  Every automorphism of $\chi$ is either elliptic or hyperbolic, so for $i=1,2$, $g_i$ is either elliptic or hyperbolic.  If one of these is hyperbolic, we are done.  If $g_1$ and $g_2$ are both elliptic, then they stabilize two adjacent vertices $v_1$ and $v_2$, respectively.  Now, suppose that $g_1$ and $g_2$ both fix a point $p \in \chi$; we can easily see that $g_1$ fixes the geodesic segment $\left[ v_1,p\right] $ and similarly that $g_2$ fixes $\left[ v_2,p\right] $.  Because every edge of $\chi$ disconnects $\chi$, at least one of $g_1$ and $g_2$ fixes the unique edge connecting the two vertices $e(v_1,v_2)$, and therefore is in $C$.  This is a contradiction and therefore $Fix(g_1)\cap Fix(g_2)=\emptyset$, hence $g_1g_2$ is a hyperbolic isometry by Lemma \ref{ProdElliptLox}.  The case of an HNN extension is similar.
		
		Next, we show that every point of $\chi$ lies on the axis for some hyperbolic isometry $g \in G$.  Suppose that $g_0 \in G$ is an element which acts as a hyperbolic isometry on $\chi$, whose axis is $\gamma_0$.  Suppose $e_0$ is any edge of $\gamma_0$.  Since $g \cdot \gamma_0$ is an axis for $gg_0g^{-1}$, it follows that $g \cdot e_0$ is contained in the axis for $gg_0g^{-1}$.  But the action of $G$ is transitive on the set of edges of $\chi$, therefore $\chi = \bigcup \limits_{g \in G}^{} g \cdot \gamma_0$.  However, the union on the right-hand side is exactly the union of the axes for the elements of $G$ which are conjugates of $g_0$, which proves our first claim.
		
		Finally, we show that every axis $\gamma$ for a hyperbolic element $g \in G$ is an axis for some $h \in H$ as follows:  Consider the cosets $H, gH, g^2H, ... , g^nH$, where $|G:H|=n$.  These cannot be all distinct, therefore, we conclude that $g^i \in H$ for some $1 \leq i \leq n$.  However, $\gamma$  is an axis for $g^i$ as well, hence our second claim has been established.  Now, we see that $\chi$ is a union of the axes for the hyperbolic elements of $H$.  From this we deduce that every edge of $\chi$ lies on an axis of a hyperbolic element of $H$ and therefore $H$ cannot leave invariant any subtree of $\chi$.  This shows that the action of $H$ on $\chi$ is minimal and that $\chi$ is, therefore, the Bass-Serre tree for $H$.
	\end{proof}

\begin{lem}\label{HNNsdp}Let $M$ be a compact 3-manifold such that $\pi_1(M)=X*_{\mathcal{T}}$, where $\mathcal{T}$ is an incompressible torus, so that $G = A*_{C}$, where $G=\pi_1(M)$ and $A = C = \mathbb{Z}^2$.  Suppose that $U$ is a finitely generated subgroup of $G$ which contains a nontrivial subnormal subgroup $N\neq\mathbb{Z}$, then $M$ has a finite cover which is a bundle over $\mathbb{S}$ with fiber a torus, and $\pi_1(T)$ is commensurable with $U$.
\end{lem}

	\begin{proof}
		By assumption, $G$ satisfies $1\rightarrow \mathbb{Z}^2 \rightarrow G \rightarrow \mathbb{Z} \rightarrow 1$, so by Theorem \ref{Stal}, $M$ is a torus bundle over $\mathbb{S}$.  However, torus bundles over the circle are geometric and the conclusion follows from Theorem \ref{geomsfiber}.
	\end{proof}

Given a graph of groups $\Gamma$, we can form a different graph of groups $\Gamma'$ by collapsing an edge in $\Gamma$ to a single vertex with vertex group $G_{v_1}*_{G_e} G_{v_2}$, or $G_{v}*_{G_e}$ if the edge collapsed joins a single vertex $v$.  To obtain a graph of groups structure, we define the inclusion maps of all the edges meeting the new vertex in $\Gamma'$ to be the compositions of the injections indicated by $\Gamma$ followed by the canonical inclusion of $G_i \hookrightarrow G_{v_1}*_{G_e} G_{v_2}$ for $i=1,2$, or $G_v \hookrightarrow G_{v}*_{G_e}$ as appropriate.  It is obvious that $\pi_1(\Gamma) \cong \pi_1(\Gamma')$.  If $\Gamma$ is a finite graph of groups, given an edge $e \in Edge(\Gamma)$, we can proceed to collapse all the edges of $\Gamma$ except for $e$ in this way to obtain a new graph of groups $\Gamma_e$ which has a single edge.  Thus, $\Gamma_e$ gives the structure of a free product with amalgamation or an HNN extension to $G$.  We shall call this process of obtaining $\Gamma_e$ from $\Gamma$ collapsing \textit{around} $e$. 

\begin{pro}\label{InfDiam}Let $\Gamma$ be a finite graph of groups with fundamental group $\pi_1(\Gamma)=G$.  Let $T$ be the Bass-Serre tree of $\Gamma$ and let $U$ be a finitely generated subgroup of $G$ such that $U\backslash T$ has infinite diameter.  Then, for some edge $e \in Edge(\Gamma)$ the quotient of the Bass-Serre tree of the graph $\Gamma_e$ by $U$ has infinite diameter.
\end{pro}

\begin{proof}
For $e \in Edge(\Gamma)$, let $Q_e$ be the graph obtained from $T$ by collapsing every edge which is not a preimage of $e$ under the quotient map $T \rightarrow G\backslash T=\Gamma$.  One easily verifies that $Q_e$ is a tree and that the action of $G$ on $T$ descends to an action of $G$ on $Q_e$.  The quotient of $Q_e$ by $G$ has a single edge and thus gives $G$ the structure of an amalgamated free product or an HNN extension.  Further, the stabilizers of the vertices of $Q_e$ are precisely conjugates of the vertex groups of $\Gamma_e$, and the stabilizers of edges are conjugates of the edge group of $\Gamma_e$: this can easily be verified by observing that for every $v' \in T$, the quotient map $T \rightarrow Q_e$ collapses the tree $T_{v'}$, which consists of all edge paths starting at $v'$ not containing preimages of $e$.  Thus, $Stab(w') \subseteq G$ for $w' \in Q_e$ is precisely the subgroup $Inv(T_{v'})\subseteq G$ which leaves $T_{v'}$ invariant.  It is easy to see that $Inv(T_{v'})$ contains every vertex stabilizer subgroup of $G$ for vertices in $T_{v'}$ and that $G$ splits as a direct product with amalgamation of $Inv(T_{v_1'})*_{Stab(e')} Inv(T_{v_2'})$ for two adjacent vertices $v_1', v_2' \in T$ joined by the edge $e'$ (or an HNN extension).  Thus, we conclude that $Q_e$ is the Bass-Serre tree for $\Gamma_e$.

Finally, to prove that $U\backslash Q_e$ has infinite diameter for some $e$, we argue by contradiction.  Suppose $U\backslash Q_e$ has finite diameter for every $e \in Edge(\Gamma)$.  First, observe that there is a bijection between the edges of $Q_e$ and the edges of $T$ which are not collapsed.  Because $Q_e$ is the Bass-Serre tree of an amalgamated free product or an HNN extension, Lemma 2.1 in \cite{Moon} shows that since $U\backslash Q_e$ is assumed to have finite diameter, it must be a finite graph of groups.  This, however, shows that there are finitely many $U$-orbits of edges in $Q_e$, hence there are finitely many $U$-orbits of edges in $T$ lying above $e \in Edge(\Gamma)$.  Because this is true by assumption for every edge $e \in Edge(\Gamma)$ and because $Edge(\Gamma)$ is finite, we conclude that there are finitely many $U$-orbits of edges in $T$.  This means that the quotient $U\backslash T$ must be a finite graph which contradicts the assumption that $U\backslash T$ has infinite diameter.  Therefore, there exists an edge $e \in Edge(\Gamma)$ such that $U\backslash Q_e$ has infinite diameter.
\end{proof}

Propositions \ref{HBassSerre} and \ref{InfDiam} allow us to prove the generalization of Theorem 2.4 in \cite{Moon} promised in the beginning of the section:
\begin{thm}\label{infdiamgr}Let $M$ be a compact 3-manifold with $\pi_1(M)=G$, and suppose that $M$ splits along an incompressible torus $\mathcal{T}$, $M=X_1 \cup_{\mathcal{T}} X_2$, or $M=X_1 \cup_{\mathcal{T}}$.  Suppose that:
	\begin{enumerate}
		\item $G$ contains a non-trivial subnormal subgroup $N=N_0\triangleleft ... \triangleleft N_{n-1} \triangleleft N_n=G$ such that $N \neq \mathbb{Z}$,
		\item at least $n$ terms in the subnormal series $N=N_0\triangleleft ... \triangleleft N_{n-1} \triangleleft N_n=G$ are finitely generated,
		\item  $G$ contains a finitely generated subgroup $U$ of infinite index in $G$ such that $N < U$.  
	\end{enumerate}
	If the graph of groups $\mathcal{U}$ corresponding to $U$ has infinite diameter, then $M$ is finitely covered by a torus bundle over $\mathbb{S}^1$ with fiber $T$, and $U$ and $\pi_1(T)$ are commensurable.
	\end{thm}

	\begin{proof}First, we note that if the splitting of $M$ induces a HNN extension structure on $G$ which is of the form $A*_{C}$ with $A = C$, then the conclusion follows from Lemma \ref{HNNsdp}, therefore in what follows we shall assume that this is not the case.
		
		From Lemma \ref{Cutoff} we conclude that our assumption on all but one of the $N_i$ being finitely generated is actually equivalent with the assumption that every $N_i$ for $i > 0$ is finitely generated.  We consider two cases: either $|N_i:N_{i-1}|<\infty$ for all $i > 0$, or there exists at least one occurrence of $|N_i:N_{i-1}|=\infty$ where $i > 0$.

\paragraph{\textit{Case 1}:} $|N_i:N_{i-1}|<\infty$ for all $i > 0$

\medskip
		
		In this case, all the $N_i$ have finite index in $G$, except for $N$, in particular $|G:N_1|<\infty$, and $N \triangleleft N_1 < G$.  Consider the finite cover $M_{N_1}$ of $M$ whose fundamental group is $N_1$.  This cover will be made up of covers for the pieces $X_1$ and $X_2$ glued along covers of the splitting torus $\mathcal{T}$.  Let $\chi$ be the Bass-Serre tree of the graph of groups induced by the splitting of $M$ along $\mathcal{T}$ which gives $G$ the structure of an amalgamated free product or HNN extension over $\mathbb{Z}^2$.  In view of Proposition \ref{HBassSerre}, $\chi$ is the Bass-Serre tree of $N_1$.  Then, $N_1$ acts on $\chi$ with quotient the graph of groups $N_1\backslash\chi$, which is also the graph of groups decomposition of $N_1$ obtained from the splitting of fundamental group of the cover $M_{N_1}$ along the fundamental groups corresponding to lifts of the splitting torus $\mathcal{T}$ of $M$.
		
		Consider now the subgroup $U \cap N_1$; we will show that $|U:U \cap N_1|<\infty$ from which it will follow that $U \cap N_1$ is a finitely generated, infinite index subgroup of $N_1$, which contains $N$:  There is a well-defined map $\psi$ from the coset space $U/U\cap N_1$ to the coset space $G/N_1$ given by $\psi(u\, U \cap N_1)=u N_1$.  This map is easily seen to be injective, hence $|U:U \cap N_1|<|G:N_1|<\infty$.  
		
		Note that $U \cap N_1$ acts on $\chi$ with an infinite diameter quotient:  The map $\left[ x \right]_{U \cap N_1\backslash\chi} \rightarrow \left[ x \right]_{U\backslash\chi}$ from $U \cap N_1\backslash\chi$ to $U\backslash\chi$ is onto.  Hence, if $U \cap N_1\backslash\chi$ had finite diameter $L$, we would conclude that any two vertices $v_1, v_2 \in U\backslash\chi$ would be at a distance at most $L$ as we can find an edge path in $U \cap N_1\backslash\chi$ of length at most $L$ between any two preimages of $v_1$ and $v_2$.  This path projects to a path of length at most $L$ between $v_1$ and $v_2$. 
		
		Since the finite index subgroup $U \cap N_1$ of $U$ also acts on $\chi$ with an infinite diameter quotient, Proposition \ref{InfDiam} shows that there is an edge in $N_1\backslash\chi$, such that the splitting of $N_1$ along the corresponding edge group has a Bass-Serre tree whose quotient by $U \cap N_1$ has infinite diameter.  The edge groups of $N_1\backslash\chi$ are all isomorphic to $\mathbb{Z}^2$ since they are the fundamental groups of lifts of the splitting torus $\mathcal{T}$ of $M$.  Therefore the cover $M_{N_1}$ is a torus sum and we can apply Theorem 2.4 in \cite{Moon} to conclude that a finite cover of $M_{N_1}$ fibers in the desired way, and that $U \cap N_1$, and therefore also $U$, is commensurable with the fundamental group of the fiber.
				
\paragraph{\textit{Case 2}:} $|N_i:N_{i-1}|=\infty$ for some $i>0$.

\medskip
		
		Let $i_0$ be the largest integer index for which $|N_{i_0}:N_{i_0-1}|=\infty$.  In this case, we consider the finite cover $M_{N_{i_0}}$ whose fundamental group is $N_{i_0}$.  Since $N_{i_0-1}$ is assumed to be finitely generated, it is also finitely presented by Theorem 2.1 in \cite{ScottCoh}. Now, $M_{N_{i_0}}$ fibers in the desired way by Theorem 3 of \cite{HempelJaco}, and further $N_{i_0-1}$ is a subgroup of finite index in $\pi_1(T)$.  Finally, we show that $U$ is commensurable with $\pi_1(T)$. Consider $U \cap N_{i_0-1}$; this group is a subgroup of $\mathbb{Z}^2$, therefore it is either trivial, $\mathbb{Z}$ or $\mathbb{Z}^2$. Since $U \cap N_{i_0-1}$ contains the non-trivial $N \neq \mathbb{Z}$, we must have $U \cap N_{i_0-1} \cong \mathbb{Z}^2$.  Because the finite cover $M_{N_{i_0}}$ of $M$ fibers over the circle, we have $|G:\pi_1(T)\rtimes \mathbb{Z}|<\infty$.  If $U \cap N_{i_0-1}$ were not of finite index in $U$, then $U$ would obviously be of finite index in $G$, which contradicts the assumptions on $U$.  Therefore, we conclude that $U$ is commensurable with $\pi_1(T)$, as desired.	
	\end{proof}

The last step towards proving our main theorem is a restatement of Theorem 2.9 in \cite{Moon}.  While the proof in \cite{Moon} only treats the case of $N$ being a normal subgroup of $\pi_1(M)$, it is obvious that it applies verbatim to the case of $N\triangleleft_{s}\pi_1(M)$.

\begin{thm}\label{findiamgr}\textbf{(Moon, Theorem 2.9 \cite{Moon}, 2005)} Let $M$ be a compact 3-manifold with $M=X_1 \cup_{\mathcal{T}} X_2$ or $M=X_1 \cup_{\mathcal{T}}$.  Suppose that $X_i$ satisfies the following condition for $i=1,2$: 
	\begin{enumerate}
		\item 	if $\pi_1(X_i)$ contains a finitely generated subgroup $U_i$ with $|\pi_1(X_i):U_i|=\infty$ such that $U_i$ contains a nontrivial subnormal subgroup $\mathbb{Z}\neq N_i \triangleleft_s \pi_1(X_i)$, then a finite cover of $X_i$ fibers over $\mathbb{S}$ with fiber a compact surface $F_i$ and $\pi_1(F_i)$ is commensurable with $U_i$,
		\item $G=\pi_1(M)$ contains a finitely generated subgroup $U$ of infinite index in $G$ which contains a nontrivial subnormal subgroup $N$ of $G$, and that $N$ intersects nontrivially the fundamental group of the splitting torus,
		\item  $N\cap\pi_1(X_i)\neq\mathbb{Z}$.
	\end{enumerate}
Suppose, further, that $G$ contains a finitely generated subgroup $U$ of infinite index in $G$ which contains a nontrivial subnormal subgroup $N$ of $G$, and that $N$ intersects nontrivially the fundamental group of the splitting torus and $N\cap\pi_1(X_i)\neq\mathbb{Z}$.  If the graph of groups $\mathcal{U}$ corresponding to $U$ is of finite diameter, then $M$ has a finite cover $\widetilde{M}$ which is a bundle over $\mathbb{S}$ with fiber a compact surface $F$, and $\pi_1(F)$ is commensurable with $U$.
\end{thm}

For the sake of brevity, let us introduce the following terminology: We shall say that a compact manifold $M$ \textit{has property (A)} if whenever $\pi_1(M)$ contains a finitely generated subgroup $U$ of infinite index, and a nontrivial subnormal subgroup $\mathbb{Z} \neq N \triangleleft_{s} G$ which has a subnormal series in which all but one terms are assumed to be finitely generated, and such that $N<U$, then $M$ fibers over $\mathbb{S}$ with fiber a compact surface $F$ such that $\pi_1(F)$ is commensurable with $U$.\

Combining Theorem \ref{infdiamgr} and Theorem \ref{findiamgr}, we obtain

\begin{thm}\label{propsprm}Let $M$ be a compact 3-manifold with $M=X_1 \cup_{\mathcal{T}} X_2$ or $M=X_1 \cup_{\mathcal{T}}$, and suppose that the following conditions are satisfied:
\begin{enumerate}
\item each $X_i$, for $i=1,2$, has property (A),
\item $G=\pi_1(M)$ contains a finitely generated subgroup $U$ of infinite index, which contains a nontrivial subnormal subgroup $N$ of $G$,
\item $N$ has a subnormal series in which the terms, except for $N$, are assumed to be finitely generated,
\item $N$ intersects nontrivially the fundamental group of the splitting torus $\mathcal{T}$,
\item  $N\cap\pi_1(X_i)\neq\mathbb{Z}$, for $i=1,2$.
\end{enumerate}
	Then, $M$ has a finite cover which is a bundle over $\mathbb{S}$ with fiber a compact surface $F$, and $\pi_1(F)$ is commensurable with $U$.
\end{thm}

\section{The main theorem}\label{MainS}
I am now ready to prove my main result.  Recall the Geometrization Theorem proved by Perelman in 2003:
	
\begin{thm}\label{GeomThm}\textbf{(Perelman, Geometrization Theorem \cite{Per1}, \cite{Per2}, \cite{Per3}, 2003)} Let $M$ be an irreducible compact 3-manifold with empty or toroidal boundary.  Then there exists a collection of disjointly embedded incompressible tori $\mathcal{T}_1, ... ,\mathcal{T}_k$ such that each component of $M$ cut along $\mathcal{T}_1 \cup ... \cup \mathcal{T}_k$ is geometric.  Furthermore, any such collection with a minimal number of components is unique up to isotopy.
\end{thm}

 To improve the exposition of the proof of the main theorem below, we define property $(A')$:  Let $M$ be a compact 3-manifold with empty or toroidal boundary which has a decomposition into geometric pieces $\mathfrak{D}=\left(M_1,...,M_2;\mathcal{T}_1,...,\mathcal{T}_p\right)$, where each $M_i$ is a compact geometric submanifold of $M$ with toroidal or empty boundary and each $T_i$ is an incompressible torus.  We shall say that $\left(M,\mathfrak{D}\right)$ has property $(A')$ if whenever $\pi_1(M)$ contains a finitely generated subgroup $U$ with $|\pi_1(M):U|=\infty$ and a nontrivial subnormal subgroup $\mathbb{Z}\neq N \triangleleft_s\pi_1(M)$ which has a subnormal series in which all but one terms are assumed to be finitely generated, such that $N<U$, and such that $N$ intersects nontrivially the fundamental groups of the splitting tori $\mathcal{T}_i$ in $\mathfrak{D}$ and $N\cap\pi_1(M_i)\neq\mathbb{Z}$ for all $X_i\in\mathfrak{D}$, then $M$ fibers over $\mathbb{S}$ with fiber a compact surface $F$ such that $\pi_1(F)$ is commensurable with U.

Using this notation, Theorem \ref{infdiamgr} and the proof of Theorem \ref{findiamgr} in \cite{Moon} together yield:

\begin{pro}\label{propsprime}Let $\left(X_1,\mathfrak{D}_1\right) $ and $\left( X_2,\mathfrak{D}_2\right) $ each be a compact manifold along with a decomposition into geometric pieces along incompressible tori.  Suppose that $\left(X_1,\mathfrak{D}_1\right) $ and $\left( X_2,\mathfrak{D}_2\right) $ have property $(A')$.  If $M = X_1\cup_{\mathcal{T}} X_2$ or $M = X_1 \cup_{\mathcal{T}}$, where $\mathcal{T}$ is an incompressible torus disjoint from the tori in $\mathfrak{D}_1$ and $\mathfrak{D}_2$, and if the tori and geometric pieces from $\mathfrak{D}_1$ and $\mathfrak{D}_2$ along with $\mathcal{T}$ together give a geometric decomposition $\mathfrak{D}$ for $M$, then $\left( M,\mathfrak{D}\right) $ has property $(A')$.
\end{pro}

\begin{thm}\label{Main}Let $M$ be a compact 3-manifold with empty or toroidal boundary.  If $G=\pi_1(M)$ contains a finitely generated subgroup $U$ of infinite index in $G$ which contains a nontrivial subnormal subgroup $N$ of $G$, then: (a) $M$ is irreducible, (b) if further: 
	\begin{enumerate}
		\item  $N$ has a subnormal series of length $n$ in which $n-1$ terms are assumed to be finitely generated,
		\item $N$ intersects nontrivially the fundamental groups of the splitting tori of some decomposition $\mathfrak{D}$ of $M$ into geometric pieces, and
		\item the intersections of $N$ with the fundamental groups of the geometric pieces are not isomorphic to $\mathbb{Z}$,
	\end{enumerate}
	 then, $M$ has a finite cover which is a bundle over $\mathbb{S}$ with fiber a compact surface $F$ such that $\pi_1(F)$ and $U$ are commensurable.
\end{thm} 

\begin{proof}First we prove (a):  By Theorem 1 in \cite{MilnorPrime} $M\cong M_1 \sharp M_2 \sharp...\sharp M_p$, where each $M_i$ is a prime manifold.  Then, we have $G=G_1*G_2*...*G_p$, where $G_i=\pi_1(M_i)$.  By Theorem 1.5 in \cite{Elkalla}, we must have $G_i={1}$ for $i \geq 2$, after possibly reindexing the terms.  Therefore, the Poincare Conjecture implies that $M_i \cong \mathbb{S}^3$ for all $i \geq 2$, and we conclude that $M$ is a prime manifold.  Therefore $M$ must be irreducible, for if it were not, then $M\cong \mathbb{S}^2\times\mathbb{S}$, hence $G=\mathbb{Z}$, $U=\left\lbrace 1\right\rbrace $, contradicting the hypothesis of the theorem.  To prove (b), we make an inductive argument to prove that every connected submanifold $M'\subseteq M$ which is a union of $X_i\in\mathfrak{D}$ has property $(A')$ with respect to its decomposition $\mathfrak{D}'$ into geometric pieces inherited from $\mathfrak{D}$.  We proceed by induction on the number $m$ of geometric pieces $X_i$ in $M'$.  It is clearly true that if $M'=X_i$ for some $i$, then $M'$ has property $(A')$ by Theorem \ref{geomsfiber}.  Suppose that all submanifolds which are a union of at most $m-1$ geometric pieces from $\mathfrak{D}$ have property $(A')$ with respect to their geometric decompositions along incompressible tori from $\mathfrak{D}$, and suppose that $M'$ is a union of $m$ geometric pieces from $\mathfrak{D}$ so that $\mathfrak{D}'=\left(X_{i_1},...,X_{i_m};\mathcal{T}_{i_1},...,\mathcal{T}_{i_s}\right)$.  If we cut $M'$ along all tori in $\mathfrak{D}'$ which are in $X_{i_m}$, $M'$ will be decomposed into connected submanifolds $N_1,...,N_w,X_{i_m}$, in such a way that each of the $N_i$ has a geometric decomposition $\mathfrak{D}_i$ which consists of at most $m-1$ pieces from $\mathfrak{D}$, and tori also from $\mathfrak{D}$.  Thus, it follows that each of the submanifolds $N_i,...,N_w, X_{i_m}$ with its geometric decomposition inherited from $\mathfrak{D}$ has property $(A')$.  Since $M'$ can be obtained from the pieces $N_1,...,N_w,X_{i_m}$ by performing a finite number of torus sums of the form $M_1\cup_{\mathcal{T}} M_2$ or $M_1\cup_{\mathcal{T}}$ for $M_i\in \left\lbrace N_1,...,N_w,X_{i_m}\right\rbrace $ and $\mathcal{T}\in\mathfrak{D}$, $M'$ also has property $(A')$ with respect to its geometric decomposition along tori from $\mathfrak{D}$ by Proposition \ref{propsprime}.  By induction, $M$ also has property $(A')$ hence it fibers in the required way.
\end{proof}

\section{Acknowledgments}\label{Ack}
I wish to thank Prof. Peter Scott, who was my Ph.D. thesis adviser at The University of Michigan in Ann Arbor, for introducing to me to the problem of the fibering of compact 3-manifolds over the circle and for the many years of patient discussions and invaluable advice.  Notably, I am indebted to Prof. Scott for sketching an idea for the proof of Proposition \ref{HBassSerre} during one of our many e-mail discussions.  I would also like to thank the Department of Mathematics at The University of Michigan for giving me the opportunity to work with Prof. Peter Scott in the capacity of a Visiting Scholar and for granting me remote access to the university's library resources in the period of 2012 to 2014.  Finally, I would like to thank the anonymous referee for the many helpful suggestions, which have greatly contributed to the readability of this research article.

\end{document}